\documentclass[final,3p,times,12pt]{elsarticle}

\usepackage{amsmath,amssymb,mathtools,amsthm}  
\usepackage{graphicx}                           
\usepackage{caption,subcaption}                
\usepackage{enumitem}                           
\usepackage{xcolor}                             
\usepackage{mathrsfs}                           

\definecolor{brown}{rgb}{0.6,0.3,0}   
\definecolor{blue}{rgb}{0,0,1}

\newtheorem{lemma}{Lemma}[section]
\newtheorem{theorem}{Theorem}[section]
\newtheorem{proposition}{Proposition}[section]
\newtheorem{corollary}{Corollary}[section]
\newtheorem{remark}{Remark}[section]

\numberwithin{equation}{section}  

\newcommand{\beq}{\begin{equation}}
	\newcommand{\eeq}{\end{equation}}

\def\a{\alpha}

\usepackage[numbers]{natbib}
\usepackage{array}
\usepackage{booktabs}

\journal{ }
\begin{document}
	
\begin{frontmatter}
\title{Separating zeros of polynomials using an added interlacing point}

\author[label1]{Kerstin Jordaan}\ead{jordakh@unisa.ac.za}
\author[label2]{Vikash Kumar}\ead{kumarv@unisa.ac.za}
\cortext[cor1]{Corresponding author}

\address[label1]{College of Economic and Management Sciences, University of South Africa, Pretoria, 0002, South Africa}
\address[label2]{Department of Decision Sciences, University of South Africa, Pretoria 0002, South Africa}

\begin{abstract}
 Following a systematic analysis of existing results, we investigate when complete interlacing between the zeros of distinct polynomial sequences, $\{\mathcal{P}_n\}$ and $\{\mathcal{G}_n\}$ can be achieved by using a naturally arising extra point. Specifically, we analyse several general mixed recurrence relations that ensure the $n+1$ zeros of the polynomial $(x-E)\mathcal{P}_n(x)$ interlace with the $k$ zeros of $\mathcal{G}_k$, where $k=n$ or $n+1$. In addition, we show that imposing specific conditions on the extra point $E$ yields full interlacing between the zeros of $\mathcal{P}_n$ and $\mathcal{G}_k$ for a suitable choice of $n$. The approach provides a consolidated framework broadly applicable to both orthogonal and non-orthogonal polynomials and we illustrate this with new interlacing results for zeros of Krawtchouk, Meixner, and Narayana polynomials. We also illustrate that this general approach can be used to recover and refine existing results regarding the complete interlacing of zeros for classical Jacobi and Laguerre polynomials.
\end{abstract}


\end{frontmatter}

\section{Introduction }
Let $p_n$ and $q_m$ be polynomials of degree $n$ and $m$, respectively, having real and distinct zeros. Let the zeros of $p_n$ be denoted by $\{x_{k,n}\}_{k=1}^n$ and the zeros of $q_m$ by $\{y_{k,m}\}_{k=1}^m$, in ascending order. We say that the zeros of $p_n$ alternate the zeros of $q_n$, denoted $p_n \prec q_n$ \cite{Liu-Wang-Unified}, if
\begin{align}
	x_{1,n} < y_{1,n} < x_{2,n} <y_{2,n}< \cdots < y_{n-1,n} < x_{n,n} < y_{n,n}.
\end{align}
We say that the zeros of $p_n$ interlace the zeros of $q_{n-1}$, denoted $p_n \prec q_{n-1}$ \cite{Liu-Wang-Unified}, if
\begin{align}
	x_{1,n} < y_{1,n-1} < x_{2,n} < y_{2,n-1} < \cdots < y_{n-1,n-1} < x_{n,n}.
\end{align}

Interlacing of zeros is a fundamental property in the theory of polynomials \cite{Fisk-2008notes, Obrechkoff-2003book,Rahman-Schmeisser-Book} and has useful applications. For example, Martínez-Finkelshtein et. al. \cite{Andrei-Rafael-Perales-Finitefree-convolution-IMRN-2024} show that the preservation of interlacing by finite free convolutions is an essential tool for establishing the real-rootedness of specific hypergeometric families and utilise this property to prove that root separation, or mesh, remains non-decreasing under additive finite free convolution.  Over the years, significant progress has been made in the study of interlacing properties, particularly for zeros of orthogonal polynomials see, for example, \cite{Brezinski-2004, Dimitrov-First-Paper-Interlacing-Notation, Dimitrov-Ismail-Rafaeli-JAT-2013,Driver-Jordaan-NumerMath-2007, Driver-Jordaan-JAT-2012,  Driver-Jordaan-Mbuyi-NumericalAlg-2008, Joulak-ANM-2005} and applications. Interlacing properties of two different polynomials of degree $n$ and $n-1$ are used by Lubinsky \cite{Lubinsky-Interlacing-and-op-ProcAMS-2016} to derive general quadrature identities and also obtain the new orthonormal polynomials sequence with respect to the specific orthogonality measure. In \cite{Levin-Lubinsky-AsymptoticOP-and-Interlacing-JAT-2022}, interlacing properties are used to study the relationship between the pointwise asymptotics of orthogonal polynomials and the spacing of zeros of orthogonal polynomials of consecutive degrees. The fundamental property that zeros of consecutive terms in a sequence of orthogonal polynomials are interlaced is an important reason for the prominence of orthogonal polynomials in these investigations, but the interlacing property is not restricted only to orthogonal polynomials. One of the main tools for establishing interlacing of zeros is the use of mixed recurrence relations and this does not require orthogonality of any of the polynomials in the relation, only interlaced zeros of two of the polynomials in the relation.

Let $\mathcal{P}_i$, $\mathcal{G}_j$, and $\mathcal{Q}_k$ denote polynomials of degrees $i$, $j$, and $k$, respectively, whose zeros are  real and  lie in an interval $(a,b)$. Assume that $\mathcal{G}_j \prec \mathcal{Q}_k$ whenever $k=j-1$ or $\mathcal{Q}_k \prec \mathcal{G}_j$ whenever $k=j+1$. These polynomials satisfy a general mixed recurrence relation
\begin{align}\label{MoreGeneralMixedTTRR}
	A(x)\mathcal{P}_i(x)=B(x)\mathcal{G}_j(x)+H(x)\mathcal{Q}_k(x),
\end{align}
where $A(x)$, $B(x)$, and $H(x)$ are continuous functions on $(a,b)$. In Table~\ref{Summary-Mixed-recurrence}, we summarize known results concerning the interlacing of zeros of two different sequences of polynomials obtained via mixed recurrence relations of type \eqref{MoreGeneralMixedTTRR} under suitable conditions.   Mixed recurrence relation of type \eqref{MoreGeneralMixedTTRR} are also used to obtain inner bounds for the extreme zeros of polynomials; see for instance \cite{Jooste-Jordaan-Extremezero-MP-JAT-2025} and references therein. 

As is evident from Table \ref{Summary-Mixed-recurrence}, the case where the coefficient $H(x)$ is linear has not been analysed in detail for different values of $i$, $j$ and $k$. The general result in \cite{Jooste-Jordaan-Numeralg-2025} for $i=n$, $j=n$, $k=n+1$ and $H(x)=x-E$, was applied to obtain new interlacing results for zeros of Meixner-Pollaczek, Pseudo-Jacobi polynomials and Continuous Hahn polynomials. Some results along these lines, proved specifically for Laguerre, Jacobi and Gegenbauer polynomials but involving the case where $H(x)=-(x-E)$, were explored by Arves\'u et. al. in \cite{Arvesu-Driver-Littlejohn-ITSF-2021} for Laguerre polynomials when $i=n$, $j=n+1$ and $k=n+1$, and in \cite{Arvesu-Driver-Littlejohn-RamanujanJ-2023}, for Jacobi and Gegenbauer polynomials when $i=n$, $j=n$ and $k=n+1$.

\begin{table}[ht]
	\centering
	\setlength{\tabcolsep}{20pt}  
	\begin{tabular}{@{} >{\raggedright\arraybackslash}p{8cm} l @{}}
		\toprule
		\textbf{Conditions} & \textbf{Interlacing} \\
		\midrule
		$i=n-l, j=n, k=n+1$	 & $\mathcal{Q}_{n+1}\prec B_l\mathcal{P}_{n-l}$ \cite[Theorem 2.1]{Driver-Jordaan-JAT-2012}\\	
		$A\neq0 ,B:=B_l(x)${, deg$(B_l)=l$, $l=1,2,\dots,n-1$}\\
		\addlinespace[8pt]
		$i=n, j=n, k=n+1$	 & $\mathcal{Q}_{n+1}\prec\mathcal{P}_n$ \cite[Lemma 1.1]{Jordaan-Tookos-ANM-2009}\\ $A=1, B\neq0,H\neq0$&$\mathcal{P}_n\prec\mathcal{G}_n$ or $\mathcal{G}_n\prec\mathcal{P}_n$ \\ 
		
		\addlinespace[10pt]
		$i=n+1, j=n, k=n+1$	 & $\mathcal{P}_{n+1}\prec\mathcal{G}_n$\cite[Lemma 1.1]{Jordaan-Tookos-ANM-2009}\\ $A=1, B\neq0,H\neq0$&$\mathcal{P}_{n+1}\prec\mathcal{Q}_{n+1}$ or $\mathcal{Q}_{n+1}\prec\mathcal{P}_{n+1}$ \\

		\addlinespace[12pt]
		$i=n, j=n, k=n-1$& $\mathcal{P}_{n}\prec\mathcal{G}_n$ \cite[Lemma 1]{Tcheutia-Jooste-Koepf-ANM-2018}\\ $A=1, H\neq0$ \\ 
		
		\addlinespace[10pt]
		$i=n, j=n, k=n+1$	& $(x-E)\mathcal{P}_{n}\prec\mathcal{G}_n$ \cite[Theorem 2.1]{Jooste-Jordaan-Numeralg-2025}\\ $A(x)>0, B(E)\neq0, H(x)=x-E,~ E\in \mathbb{R}$ \\ 
		
		\bottomrule
	\end{tabular}
	\caption{\label{Summary-Mixed-recurrence}Summary of results on interlacing of zeros for polynomials satisfying mixed recurrence equation $A(x)\mathcal{P}_i(x)=B(x)\mathcal{G}_j(x)+H(x)\mathcal{Q}_k(x)$ with $\mathcal{Q}_{n+1}\prec\mathcal{G}_n$ or $\mathcal{G}_n\prec\mathcal{Q}_{n-1}$}
\end{table}

In this paper we address general cases where $H(x)$ has a single sign change in \S \ref{main results}, providing proof for these results in \S \ref{Proofs}.  We use Theorem \ref{MainTheorem1} to establish new interlacing results for zeros of Krawtchouk polynomials in \S \ref{Examples Krawtchouk} and Meixner polynomials in \S \ref{Examples Meixner}. We apply Theorem \ref{MainTheorem2*} to a class of non-orthogonal polynomials, known as Narayana polynomials in \S \ref{Examples Naryana}. In \S \ref{Examples Jacobi}, we prove a stronger result than \cite[Theorem 3]{Arvesu-Driver-Littlejohn-RamanujanJ-2023} for Jacobi polynomials using Theorem \ref{MainTheorem2*}, also recovering \cite[Theorem 1]{Arvesu-Driver-Littlejohn-RamanujanJ-2023} using Theorem \ref{MainTheorem1}. Finally, in \S \ref{Examples Laguerre}, we use Theorem \ref{MainTheorem1} to recover \cite[Theorem 2.1]{Arvesu-Driver-Littlejohn-ITSF-2021} for Laguerre polynomials.

\section{Interlacing of zeros}\label{main results}

\begin{theorem}\label{MainTheorem1}
	Let $\mathcal{G}_{n+1}$ and $\mathcal{Q}_{n+1}$ denote monic polynomials of degree $n+1$ with all real zeros on an (infinite or finite) interval $(a,b)$. Denote the zeros of $\mathcal{G}_{n+1}$ by $\{x_{k,n+1}\}_{k=1}^{n+1}$ and suppose that $\mathcal{G}_{n+1}\prec \mathcal{Q}_{n+1}$ or $\mathcal{Q}_{n+1} \prec \mathcal{G}_{n+1}$. Let $\mathcal{P}_{n}$ be any monic polynomial with $n$ real zeros, which is written in terms of $\mathcal{G}_{n+1}$ and $\mathcal{Q}_{n+1}$, for each $n\in\mathbb{N}$, as follows 
	
	\begin{align}\label{GeneralMixedTTRR}
		A(x)\mathcal{P}_{n}(x)=B(x)\mathcal{G}_{n+1}(x)+(x-E)\mathcal{Q}_{n+1}(x),
	\end{align}
	where $E\in \mathbb{R}$ and $A(x)>0$ on the interval $(a,b)$. Assume also that $B(E)\neq0$ and that the polynomials $\mathcal{G}_{n+1}$ and $\mathcal{P}_{n}$ have no common zeros.
	\begin{enumerate}
		\item If $\mathcal{Q}_{n+1} \prec \mathcal{G}_{n+1}$, then $E<x_{n+1,n+1}$  and $(x-E)\mathcal{P}_n(x) \prec \mathcal{G}_{n+1}(x)$. Moreover, $\mathcal{G}_{n+1}\prec\mathcal{P}_n$  if and only if $E<x_{1,n+1}$.
		\item If $\mathcal{G}_{n+1}\prec \mathcal{Q}_{n+1}$, then  $E>x_{1,n+1}$ and $\mathcal{G}_{n+1}(x)\prec (x-E)\mathcal{P}_{n}(x)$. Moreover, $\mathcal{G}_{n+1}\prec\mathcal{P}_n$ if and only if $E>x_{n+1,n+1}$. 
	\end{enumerate}
	
\end{theorem}

\begin{theorem}\label{MainTheorem2*}
	Suppose the zeros of the monic polynomials $\mathcal{G}_{n}$, $\mathcal{P}_n$, and $\mathcal{Q}_{n-1}$, with degree $n,n$ and $n-1$ respectively, are all real on an (infinite or finite) interval $(a,b)$. Denote the zeros of $\mathcal{G}_{n}$ by $\{x_{k,n}\}_{k=1}^{n}$ in ascending order. Assume that $\mathcal{G}_n\prec \mathcal {Q}_{n-1}$
	and that there exist a real constant $E$ and polynomials $A(x)$, $B(x)$ with $A(x)>0$ on $(a,b)$ such that
	\begin{align}\label{GeneralMixedTTRR2*}
		A(x)\mathcal{P}_{n}(x) = B(x)\mathcal{G}_{n}(x) - (x-E)\mathcal{Q}_{n-1}(x).
	\end{align}
	Assume also that $B(E) \neq 0$ and that $\mathcal{G}_{n}$ and $\mathcal{P}_{n}$ have no common zeros.	Then, $(x-E)\mathcal{P}_n(x) \prec \mathcal{G}_n(x)$. In particular, $\mathcal{P}_n\prec \mathcal{G}_n$, whenever $E > x_{n,n}$ while $\mathcal{G}_n \prec \mathcal{P}_n$, whenever $E < x_{1,n}$.
\end{theorem}

The following result is an exact generalisation of \cite[Theorem~3]{Arvesu-Driver-Littlejohn-RamanujanJ-2023}, which was written only for Jacobi polynomials (both consider a mixed recurrence relation with $i=n$, $j=n$, $k=n+1$ and $H(x)=-(x-E)$). Note that, since we obtain a much stronger result than \cite[Theorem~3]{Arvesu-Driver-Littlejohn-RamanujanJ-2023} for interlacing of Jacobi polynomials in \S \ref{Examples Jacobi} using Theorem \ref{MainTheorem2*}, we include Theorem \ref{MainTheorem2} here only for the sake of completeness.  

\begin{theorem}\label{MainTheorem2}
	Suppose the zeros of the monic polynomials $\mathcal{G}_{n}$, $\mathcal{P}_n$, and $\mathcal{Q}_{n+1}$, with degrees $n,n$ and $n+1$ respectively, are all real on an (infinite or finite) interval $(a,b)$. Denote the zeros of $\mathcal{G}_{n}$ and $\mathcal{P}_n$ in ascending order by 
	$\{x_{k,n}\}_{k=1}^{n}$ and $\{z_{k,n}\}_{k=1}^{n}$, respectively.  Assume that $\mathcal{Q}_{n+1}\prec \mathcal{G}_{n}$ and the following relation holds
	\begin{align}\label{GeneralMixedTTRR2}
		A(x)\mathcal{P}_{n}(x) = B(x)\mathcal{G}_{n}(x) - (x-E)\mathcal{Q}_{n+1}(x),
	\end{align}
	where $E \in \mathbb{R}$ and $A(x) > 0$ on the interval $(a,b)$. 
	Assume also that $B(E) \neq 0$ and that the polynomials $\mathcal{G}_{n}$ and $\mathcal{P}_{n}$ have no common zeros. 
	Then, at least $n-2$ zeros of $\mathcal{P}_n$ lie in distinct intervals whose end points are consecutive zeros of $\mathcal{G}_n$.  
	Depending on the position of the point $E$, the remaining two zeros of $\mathcal{P}_n$ gives the following
	\begin{enumerate}
		\item $(x-E)\mathcal{G}_n(x) \prec \mathcal{P}_n(x)$ whenever $
		x_{k',n} < z_{k',n} < E < z_{k'+1,n} < x_{k'+1,n}
		$
		for some $k' \in \{1, 2, \ldots, n-1\}$.
		
		\item $\mathcal{P}_n\prec\mathcal{G}_n$ whenever $
		E < x_{1,n}$ or $\mathcal{G}_n\prec\mathcal{P}_n$ whenever $
		E > x_{n,n}$.
	\end{enumerate}
\end{theorem}

\section{Examples}\label{Examples}

\subsection{Krawtchouk polynomials}\label{Examples Krawtchouk}
Monic Krawtchouk polynomials, denoted by $K_n(x; p, N)$, are defined in terms of the Gauss hypergeometric function as
\begin{equation}\label{Monic-Krawtchouk-2F1}
	K_n(x; p, N) = (-N)_n p^n \, {_2F_1}\left(\begin{matrix} -n, -x \\ -N \end{matrix}; \frac{1}{p}\right).
\end{equation}

These polynomials are orthogonal with respect to the discrete weight function supported on the set $\{0, 1, \dots, N\}$, which is given by the binomial distribution
\begin{equation}\label{Weightfunction-Krawtchouk}
	w(x;p,N) = \binom{N}{x} p^x (1-p)^{N-x}, \quad x \in \{0, 1, \dots, N\},
\end{equation}
with parameters $N \in \mathbb{N}$, $0 < p < 1$, and $0 \leq n \leq N$.
These monic polynomials satisfy the following three-term recurrence relation (cf. \cite[page 237]{Koekoek-Book-HypergeometricOP})
\begin{align} \label{Monic-Krawtchouk-:ttrr}
	\nonumber	K_{n+2}(x;p, N+1) &= (x - (n+1)(1-p) -(N-n)p) K_{n+1}(x; p,N+1) \\
	&\hspace{3cm}- (n+1)p(1-p)(N+1-n) K_n(x; p,N+1)
\end{align}
Next, we derive a mixed recurrence relation for monic Krawtchouk polynomials with parameters $N$ and $N+1$. This relation explicitly identifies a linear factor governing the interlacing of zeros. 

\begin{lemma}
	Let $K_n(x; p, N)$ denote monic Krawtchouk polynomials of degree $n$ with parameters $0 < p < 1, N\in \mathbb{N}$ and $n\leq N$. Then the  following mixed recurrence relation holds	\begin{multline}\label{Mixed-TTRR-linear-factor-interlacing-Krawtchouk-n(N+1)-(n+1)N}
		p(1-p)(n+1)(N+1-n)K_n(x; p, N+1) \\
		= (N+1-x)K_{n+1}(x; p, N) + \Big(x - (N+1) + p(n+1)\Big)K_{n+1}(x; p, N+1).
	\end{multline}
\end{lemma}

\begin{proof}
	We use the Christoffel transformation \cite[page 35]{Chiharabook} connecting Krawtchouk measures with parameters $N$ and $N+1$. Given that the weight function corresponding to Krawtchouk  polynomials  satisfies
	\begin{equation*}
		w(x; p, N) = \frac{N+1-x}{(N+1)(1-p)} w(x; p, N+1),
	\end{equation*}
	the monic polynomial of degree $n+1$ evaluated at $a = N+1$, orthogonal with respect to $w(x; p, N)$, is expressed as 	
	$\displaystyle 	K_{n+1}(x; p, N) $
	\begin{equation}\label{Kernel-Krawtchouk}
		= \frac{1}{N+1-x} \left( -K_{n+2}(x; p, N+1) + \frac{K_{n+2}(N+1; p, N+1)}{K_{n+1}(N+1; p, N+1)} K_{n+1}(x; p, N+1) \right).
	\end{equation}
	Employing the evaluation
	\begin{equation*}
		K_k(N+1; p, N+1) = k! \binom{N+1}{k} (1-p)^k,
	\end{equation*}
	this simplifies  \eqref{Kernel-Krawtchouk} and we may write the relation as
	\begin{equation} \label{eq:christoffel}
		(N+1-x) K_{n+1}(x; p, N) = (N-n)(1-p) K_{n+1}(x; p, N+1) - K_{n+2}(x; p, N+1).
	\end{equation}
	Substituting \eqref{Monic-Krawtchouk-:ttrr} into \eqref{eq:christoffel} to eliminate $K_{n+2}$ and grouping the terms for $K_{n+1}(x;p, N+1)$ and $K_n(x; p,N+1)$ gives
	\begin{align*}
		\nonumber	&(N+1-x) K_{n+1}(x;p, N) = (n+1)p(1-p)(N+1-n) K_n(x;p, N+1) \\&\hspace{1.5cm}+ \Big[ (N-n)(1-p) - x + (n+1)(1-p) +(N-n)p)  \Big] K_{n+1}(x;p, N+1).
	\end{align*}
	Finally, we simplify the bracketed term $[\cdot]$ and rearrange to obtain \eqref{Mixed-TTRR-linear-factor-interlacing-Krawtchouk-n(N+1)-(n+1)N}.
\end{proof}
Note that the mixed recurrence relation \eqref{eq:christoffel}, with the index
$n+1$ replaced by $n$, satisfies all the hypotheses listed in the second row of
Table~\ref{Summary-Mixed-recurrence}. Hence, as proved in \cite{Jordaan-Tookos-ANM-2009}, $K_{n+1}(x;p,N+1)\prec K_{n}(x;p,N)$
for any $0<p<1$, $N\in\mathbb{N}$,
and $n\in\{0,1,\ldots,N\}$.
We next examine the conditions under which the
zeros of $K_{n}(x;p,N+1)$ and $K_{n+1}(x;p,N)$ also interlace.

\begin{corollary}
	Let $\{x_{k,n+1}\}_{k=1}^{n+1}$ denote the zeros of the monic Krawtchouk polynomial $K_{n+1}(x; p, N)$ defined in \eqref{Monic-Krawtchouk-2F1}, in ascending order.  Then, $x_{1,n+1}<E_{n,p,N}:= N+1-p(n+1)$ and for any $0<p<1, N\in \mathbb{N}$ and $n\in \{0,1,...,N\}$, $K_{n+1}(x; p, N)\prec (x-E_{n,p,N})K_n(x; p, N+1)$. Moreover, $K_{n+1}(x; p, N)\prec K_n(x; p, N+1)$ if and only if the largest zero of $K_{n+1}(x; p, N)$ is less than the point $E_{n,p,N}$ (i.e. $x_{n+1,n+1}<E_{n,p,N}$).
\end{corollary} 
\begin{proof}
	Let $\{y_{k,n+1}\}_{k=1}^{n+1}$ denote the zeros of the polynomial 
	$K_{n+1}(x; p, N+1)$. 
	The monic Krawtchouk polynomial $K_n(x; p, N)$ satisfies the mixed three-term recurrence 
	relation given in \eqref{Mixed-TTRR-linear-factor-interlacing-Krawtchouk-n(N+1)-(n+1)N}.	Define
	\begin{align}
		\mathcal{G}_{n+1}(x):=K_{n+1}(x; p, N), ~
		\mathcal{Q}_{n+1}(x):=K_{n+1}(x; p, N+1), ~
		\mathcal{P}_n(x):=K_n(x; p, N),
	\end{align}and 
	\begin{align}
		A(x):=p(1-p)(n+1)(N+1-n), \qquad 
		B(x):=N+1-x.
	\end{align}
	Clearly, $A(x)>0$ and $B(E_{n,p,N})\neq 0$ for all $ n \le N$ and $0<p<1$.
	Furthermore, $K_{n+1}(x; p, N)\prec K_{n+1}(x; p, N+1)$ for any $N\in\mathbb{N}$, $0<p<1$, and $0\le n \le N$ \cite[Theorem 4.1]{Jordaan-Tookos-ANM-2009}.
	Hence, all the hypotheses of Theorem~\ref{MainTheorem1} are satisfied, 
	and the result follows directly from Theorem~\ref{MainTheorem1}(2).
\end{proof}

\subsection{Meixner polynomials} \label{Examples Meixner} We define monic Meixner polynomials, denoted by $M_n(x; t, w)$, using the hypergeometric representation as follows
\begin{equation}
	M_n(x; t, w) = \frac{(t)_n w^n}{(w-1)^n} \, {}_2F_1\left(\begin{matrix} -n, -x \\ t \end{matrix} ; 1 - \frac{1}{w}\right),
\end{equation}
where the parameter $t > 0$ and $0<w<1$. These polynomials form an orthogonal system with respect to the discrete weight function 
\begin{equation}\label{Meixner-Weight}
	\rho(x; t, w) = \frac{(t)_x w^x}{x!},
\end{equation} supported on the non-negative integers $x \in \{0, 1, 2, \dots\}$. Monic Meixner polynomials satisfy the three-term recurrence relation (cf. \cite[page 234]{Koekoek-Book-HypergeometricOP})

$\displaystyle  M_{n+2}(x; t, w)$
\begin{align} \label{TTRR-Monic-Meixner}
	=
	\left(x+\frac{n+1+w(n+t+1)}{w-1}\right) M_{n+1}(x; t, w) -\frac{w(n+1)(n+t)}{(w-1)^2} M_{n}(x; t, w).
\end{align}
We now derive a mixed recurrence relation using the standard three-term recurrence \eqref{TTRR-Monic-Meixner} and the Christoffel-transformed polynomials. This relation plays a crucial role in establishing the conditions under which the zeros of $M_{n}(x; t, w)$ and $M_{n+1}(x; t+1, w)$  interlace.

\begin{lemma}
	Let  $M_{n}(x; t, w)$ denote the monic Meixner polynomial of degree $n$ with parameters $t>0$ and $0<w<1$.	These polynomials satisfy the following mixed recurrence relation
	
	$\displaystyle \frac{w(n+1)(n+t)}{(w-1)^2} M_n(x; t, w)$
	\begin{align} \label{Mixed-TTRR-linear-factor-interlacing-Meixner-n(w)-(n+1)(w+1)}
		=
		- (x+t) M_{n+1}(x; t+1, w) +\left(x+t + \frac{w(n+1)}{w-1} \right) M_{n+1}(x; t, w) 
	\end{align}
\end{lemma}
\begin{proof}
	The Christoffel transformation \cite[page 35]{Chiharabook} applied to the weight \eqref{Meixner-Weight} at the point $a = -t$ yields
	\begin{align}\label{Christoffel-Weight-Meixner}
		\rho(x; t+1, w) = \frac{x + t}{t} \rho(x; t, w).
	\end{align}
	The corresponding Christoffel-transformed polynomials are given by
	\begin{align*}
		M_{n+1}(x; t+1, w) = \frac{1}{x + t} \left[ M_{n+2}(x; t, w) - \frac{n + t + 1}{w - 1} M_{n+1}(x; t, w) \right].
	\end{align*}
	Substituting the three-term recurrence \eqref{TTRR-Monic-Meixner} into this relation to eliminate the term of degree $n+2$, and subsequently collecting coefficients of $M_{n+1}(x; t, w)$ and $M_n(x; t, w)$, we arrive at the mixed recurrence \eqref{Mixed-TTRR-linear-factor-interlacing-Meixner-n(w)-(n+1)(w+1)}. This completes the proof.
\end{proof}
\begin{corollary}
	Let $\{x_{k,n+1}\}_{k=1}^{n+1}$ denote the zeros of $M_{n+1}(x; t+1, w)$ in ascending order. Define
	$
	E_{n,t,w} := -t + \frac{w(n+1)}{1-w}.
	$
	Then, for any $t > 0$ and $0 < w < 1$, we have $x_{n+1,n+1} > E_{n,t,w}$, and $(x - E_{n,t,w}) M_{n}(x; t, w)\prec M_{n+1}(x; t+1, w)$. In particular, $ M_{n+1}(x; t+1, w)\prec M_{n}(x; t, w)$ if and only if  $x_{1,n+1}>E_{n,t,w}$.
\end{corollary} 
\begin{proof}
	Monic Meixner polynomials $M_n(x; t, w)$ satisfy the mixed three-term recurrence 
	relation given in \eqref{Mixed-TTRR-linear-factor-interlacing-Meixner-n(w)-(n+1)(w+1)}.	Define
	\begin{align}
		\mathcal{G}_{n+1}(x):=M_{n+1}(x; t+1, w), ~
		\mathcal{Q}_{n+1}(x):=M_{n+1}(x; t, w), ~
		\mathcal{P}_n(x):=M_n(x; t, w),
	\end{align}and 
	\begin{align}
		A(x):=\frac{w(n+1)(n+t)}{(w-1)^2}>0, \qquad 
		B(x):=-(x+t).
	\end{align}
	Clearly, $B(E_{n,t,w})\neq 0$ for all $ t >0$ and $0<w<1$.
	Furthermore, as established in \cite[Corollary 2.2]{Jordaan-Tookos-ANM-2009} using Markov's monotonicity theorem, $M_{n+1}(x; t, w) \prec M_{n+1}(x; t+1, w)$ for $0 < w < 1$ and $t > 0$.
	Hence, all the hypotheses of Theorem~\ref{MainTheorem1} are satisfied, 
	and the result follows directly from Theorem~\ref{MainTheorem1}(1).
\end{proof}

\subsection{Narayana Polynomials}\label{Examples Naryana}
Narayana polynomials, denoted by $N_n(x)$, are given by \cite{Kostov-Finkelshtein-Shapiro-NarayanaNumbers-JAT-2009,Liu-Wang-Unified}
\begin{equation}
	N_n(x) = \sum_{k=1}^n c_{n,k} x^k, \quad \text{where } c_{n,k} = \frac{1}{n}\binom{n}{k}\binom{n}{k-1}.
\end{equation}
The coefficients $c_{n,k}$ are known as the Narayana numbers. We define reduced Narayana polynomials~\cite{Kostov-Finkelshtein-Shapiro-NarayanaNumbers-JAT-2009} as
\begin{equation}
	\mathcal{N}_n(x) = \frac{N_n(x)}{x} = \sum_{j=0}^{n-1}c_{n, j+1} x^j.
\end{equation}
The polynomial $\mathcal{N}_n(x)$ is of degree $n-1$, and $x=0$ is not a zero for any $n\geq1$. Reduced Narayana polynomials are reciprocal, satisfying the identity 
\begin{align}\label{reciprocal-identity-ReducedNarayana}
	\mathcal{N}_n(x) = x^{n-1}\mathcal{N}_n\left(\frac{1}{x}\right). 
\end{align}
Furthermore, the zeros of $\mathcal{N}_n(x)$ lie in the interval $(-\infty,0)$~\cite[Corollary 7]{Kostov-Finkelshtein-Shapiro-NarayanaNumbers-JAT-2009}.

We define the Christoffel-transformed polynomial related to reduced Narayana polynomials as
\begin{equation}\label{CT-Narayana-Polynomial}
	\tilde{\mathcal{N}}_n(x) = \sum_{j=0}^{n-1} \tilde{c}_{n,j} x^j= \frac{\mathcal{N}_{n+1}(x) - \rho_n \mathcal{N}_n(x)}{x - 1}, \quad \text{with } \rho_n = \frac{\mathcal{N}_{n+1}(1)}{\mathcal{N}_n(1)}.
\end{equation}
This example is of particular interest because reduced Narayana polynomials fall outside the classical orthogonal framework considered in the previous subsections. 

Subsequently, we demonstrate that the interlacing property is preserved under the perturbation in the coefficients of reduced Narayana polynomials defined in \eqref{CT-Narayana-Polynomial}. Specifically, we show that the zeros of reduced Narayana polynomials interlace with the zeros of the polynomials generated by these perturbed coefficients.

\begin{corollary}\label{cor:interlace_narayana}
	Let $\mathcal{N}_{n}(x)$ be the monic reduced Narayana polynomial of degree $n-1$, and let $\tilde{\mathcal{N}}_{n}(x)$ be the polynomial associated with the perturbed coefficients defined in \eqref{CT-Narayana-Polynomial}. Then, $\tilde{\mathcal{N}}_{n}(x)\prec \mathcal{N}_{n}(x)$.
\end{corollary}

\begin{proof}
	Recall the three-term recurrence relation for standard Narayana polynomials $N_n(x)$ given in \cite[Eq.~3.3]{Liu-Wang-Unified}
	\begin{equation}\label{TTRR-Narayana}
		(n+2)N_{n+1}(x) = (2n+1)(x+1)N_n(x) - (n-1)(x-1)^2 N_{n-1}(x).
	\end{equation}
	Substituting $N_k(x) = x\mathcal{N}_k(x)$ for $k=n-1, n, n+1$, we obtain the recurrence for the reduced polynomials
	\begin{equation}
		\label{eq:reduced_std_recurrence}
		(n+2)\mathcal{N}_{n+1}(x) = (2n+1)(x+1)\mathcal{N}_n(x) - (n-1)(x-1)^2 \mathcal{N}_{n-1}(x).
	\end{equation}
	Evaluating \eqref{eq:reduced_std_recurrence} at $x=1$ yields
	\begin{equation*}
		\rho_n = \frac{\mathcal{N}_{n+1}(1)}{\mathcal{N}_n(1)} = \frac{2(2n+1)}{n+2}.
	\end{equation*}
	Substituting this value of $\rho_n$ into \eqref{CT-Narayana-Polynomial}, we have
	\begin{equation}\label{CT-Narayana-with-rho}
		(x-1)\tilde{\mathcal{N}}_{n}(x) = \mathcal{N}_{n+1}(x) - \frac{2(2n+1)}{n+2}\mathcal{N}_n(x).
	\end{equation}
	Next, we substitute the expression for $\mathcal{N}_{n+1}(x)$ from \eqref{eq:reduced_std_recurrence} into \eqref{CT-Narayana-with-rho}
	\begin{align*}
		(x-1)\tilde{\mathcal{N}}_{n}(x) &= \left[ \frac{2n+1}{n+2}(x+1)\mathcal{N}_n(x) - \frac{n-1}{n+2}(x-1)^2 \mathcal{N}_{n-1}(x) \right] - \frac{2(2n+1)}{n+2}\mathcal{N}_n(x).
	\end{align*}
	Grouping the terms containing $\mathcal{N}_n(x)$ and simplifying, we obtain
	\begin{equation}\label{MixedRR-Narayana-polynomials}
		\frac{n+2}{n-1}\tilde{\mathcal{N}}_n(x) = \frac{2n+1}{n-1}\mathcal{N}_n(x) - (x-1)\mathcal{N}_{n-1}(x).
	\end{equation}
	It is known that $\mathcal{N}_n(x)\prec \mathcal{N}_{n-1}(x)$  \cite[Corollary 7]{Kostov-Finkelshtein-Shapiro-NarayanaNumbers-JAT-2009}. Setting $\mathcal{G}_{n}(x) := \mathcal{N}_n(x)$, $\mathcal{Q}_{n-1}(x) := \mathcal{N}_{n-1}(x)$, and $\mathcal{P}_n(x) := \tilde{\mathcal{N}}_n(x)$, equation \eqref{MixedRR-Narayana-polynomials} satisfies the conditions required by Theorem \ref{MainTheorem2*}. Hence, $(x-1)\tilde{\mathcal{N}}_{n}(x)\prec \mathcal{N}_{n}(x)$ and the result follows.
\end{proof}
Using the relation given in \eqref{MixedRR-Narayana-polynomials}, we can derive a compact form for the coefficients $\tilde{c}_{n,j}$ of the polynomial $\tilde{\mathcal{N}}_n(x)$ expressed in terms of the Narayana numbers.

\begin{proposition}
	Let $\tilde{\mathcal{N}}_n(x)$ be the polynomial of degree $n-1$ defined in \eqref{CT-Narayana-Polynomial}. The coefficients $\tilde{c}_{n,j}$ satisfy the following relation
	\begin{equation}\label{coefficient-interms-Narayana}
		\tilde{c}_{n,j} = \frac{3n - 2j}{n + 2} c_{n, j+1}.
	\end{equation}
\end{proposition}

\begin{proof}
	By comparing the coefficients of $x^j$ on both sides of \eqref{MixedRR-Narayana-polynomials} and applying the boundary conditions $c_{n-1,0}=0$ and $c_{n-1,n}=0$, we obtain the equation
	\begin{equation}\label{CompairingCoeff-eq1}
		\frac{n+2}{n-1} \tilde{c}_{n,j} = \frac{2n+1}{n-1} c_{n,j+1} - (c_{n-1, j} - c_{n-1, j+1}).
	\end{equation}
	We observe the following identities relating the coefficients $c_{n-1,k}$ and $c_{n, j+1}$
	\begin{align}
		\frac{c_{n-1,j+1}}{c_{n,j+1}} &= \frac{\frac{1}{n-1}\binom{n-1}{j+1}\binom{n-1}{j}}{\frac{1}{n}\binom{n}{j+1}\binom{n}{j}} = \frac{(n-j-1)(n-j)}{n(n-1)}  \notag \\
		\implies c_{n-1, j+1} &= \frac{(n-j)(n-j-1)}{n(n-1)} c_{n,j+1},
	\end{align}
	and similarly,
	\begin{align}
		\frac{c_{n-1,j}}{c_{n,j+1}} &= \frac{\frac{1}{n-1}\binom{n-1}{j}\binom{n-1}{j-1}}{\frac{1}{n}\binom{n}{j+1}\binom{n}{j}} = \frac{n}{n-1} \cdot \frac{j+1}{n} \cdot \frac{j}{n} \notag \\
		\implies c_{n-1, j} &= \frac{j(j+1)}{n(n-1)} c_{n,j+1}.
	\end{align}
	Substituting these expressions into the coefficient equation \eqref{CompairingCoeff-eq1}, we have
	\begin{align*}
		\frac{n+2}{n-1} \tilde{c}_{n,j} &= c_{n,j+1} \left[ \frac{2n+1}{n-1} - \frac{j(j+1)}{n(n-1)} + \frac{(n-j)(n-j-1)}{n(n-1)} \right] \notag \\
		&= \frac{c_{n,j+1}}{n(n-1)} \left[ n(2n+1) - j(j+1) + (n-j)(n-j-1) \right].
	\end{align*}
	Simplifying the  above equation  leads to the desired expression \eqref{coefficient-interms-Narayana}.
\end{proof}


\begin{corollary}Let $\mathcal{N}_n(x)$ be the reduced Narayana polynomial of degree $n-1$. Let $\mathcal{P}_n(x) = \sum_{j=0}^{n-1} d_{n,j} x^{j}$ be the polynomial defined by the coefficients
	\begin{align}\label{modified-naryana-coefficients}
		d_{n,j} = \binom{n-1}{j}^2 + \binom{n-1}{j+1}\binom{n-1}{j-1}
	\end{align}
	Then, for any odd integer $n \ge 3$, $(x+1)\mathcal{P}_n(x)\prec \mathcal{N}_n(x)$, while for every even integer $n \geq 2$, $\mathcal{P}_n(x)\prec \frac{\mathcal{N}_n(x)}{x+1}$.
\end{corollary}

\begin{proof}
	We first establish the polynomial identity that relates $\mathcal{N}_n(x)$, $\mathcal{N}_{n-1}(x)$, and $\mathcal{P}_n(x)$.
	Let $c_{n,j+1} = \frac{1}{n}\binom{n}{j+1}\binom{n}{j}$ be the coefficient of $x^{j}$ in $\mathcal{N}_n(x)$. 
	
	We apply Pascal’s Identity, $\binom{n}{r} = \binom{n-1}{r} + \binom{n-1}{r-1}$, to both binomial terms in the product
	
	\begin{align}
		n c_{n,j+1} = \left[ \binom{n-1}{j+1} + \binom{n-1}{j} \right] \cdot \left[ \binom{n-1}{j} + \binom{n-1}{j-1} \right].
	\end{align}
	Expanding this product yields 
	\begin{align}
		n c_{n,j+1} = \underbrace{\left[ \binom{n-1}{j}^2 + \binom{n-1}{j+1}\binom{n-1}{j-1} \right]}_{d_{n,j}} + \left[ \binom{n-1}{j}\binom{n-1}{j-1} + \binom{n-1}{j+1}\binom{n-1}{j} \right]. 
	\end{align}

	We now reconstruct the polynomials by multiplying the above equation by $x^{j}$ and summing over $j=0$ to $n-1$
	\begin{align}\label{eq1Narayana-MTTRR}
		\sum_{j=0}^{n-1} n c_{n,j+1} x^{j} = \sum_{j=0}^{n-1} d_{n,j} x^{j} + \sum_{j=0}^{n-1} \binom{n-1}{j+1}\binom{n-1}{j} x^{j} + \sum_{j=0}^{n-1} \binom{n-1}{j}\binom{n-1}{j-1} x^{j}. 
	\end{align}
	The second term in the right-hand  side of  \eqref{eq1Narayana-MTTRR} can be written as 
	\begin{align}\label{eq1Narayana-need}
		\sum_{j=0}^{n-1} \binom{n-1}{j+1}\binom{n-1}{j} x^{j} = (n-1)\mathcal{N}_{n-1}(x).
	\end{align}
	Similarly, the third term in the right-hand side of \eqref{eq1Narayana-MTTRR} can be written as 
	\begin{align}\label{eq2Narayana-need}
		\sum_{j=0}^{n-1} \binom{n-1}{j}\binom{n-1}{j-1} x^{j}  
		&= x \sum_{j=0}^{n-2} \binom{n-1}{j+1}\binom{n-1}{j} x^{j}=x (n-1)\mathcal{N}_{n-1}(x). 
	\end{align}
	Substituting \eqref{eq1Narayana-need} and \eqref{eq2Narayana-need} in \eqref{eq1Narayana-MTTRR}, we obtain 
	\begin{align*}
		n \mathcal{N}_n(x) = \mathcal{P}_n(x) + (n-1)\mathcal{N}_{n-1}(x) + x(n-1)\mathcal{N}_{n-1}(x).
	\end{align*}
	Rearranging the term, we have 	
	\begin{align}\label{MixedRR2-Narayana}
		\frac{1}{n-1} \mathcal{P}_n(x) = \frac{n}{n-1} \mathcal{N}_n(x) - (x+1)\mathcal{N}_{n-1}(x).
	\end{align}
	It is known that $\mathcal{N}_n(x)\prec \mathcal{N}_{n-1}(x)$ \cite[Corollary 7]{Kostov-Finkelshtein-Shapiro-NarayanaNumbers-JAT-2009}. By the reciprocal property  \eqref{reciprocal-identity-ReducedNarayana} of reduced Narayana polynomials,  we see that $x=-1$ is a zero of polynomials $\mathcal{N}_n(x)$ for any even integer $n\geq2$.  Hence, for any odd integer $n\geq3$, $\mathcal{P}_n$ and $\mathcal{N}_n$ share no common zero. Consequently, the identity \eqref{MixedRR2-Narayana} satisfies the conditions of Theorem \ref{MainTheorem2*}, and the result for odd degree follows. Moreover, for every even integer $n \geq 2$, the relation \eqref{MixedRR2-Narayana} implies that $\mathcal{P}_n(x)$ and $\mathcal{N}_n(x)$ share a common zero at $x=-1$. Consequently, for every even integer $n \geq 2$, the zeros of $\mathcal{P}_n(x)$, defined in \eqref{modified-naryana-coefficients}, interlace the zeros of the quotient polynomial $\frac{\mathcal{N}_n(x)}{x+1}$.
\end{proof}

\subsection{Jacobi polynomials}\label{Examples Jacobi}
Monic Jacobi polynomials \( P_n^{(\alpha, \beta)}(x) \) form an orthogonal sequence on the interval \( (-1, 1) \) with respect to the measure $d\mu(x) = (1 - x)^\alpha (1 + x)^\beta \, dx,$ where \( \alpha > -1 \) and \( \beta > -1 \) \cite{Chiharabook}. These polynomials satisfy the following three-term recurrence relation 
\begin{align} \label{TTRR-Monic-Jacobi}
	P_{n+1}^{(\alpha, \beta)}(x) &= (x - c^{(\alpha, \beta)}_{n+1}) P_n^{(\alpha, \beta)}(x) - \lambda^{(\alpha, \beta)}_{n+1}  P_{n-1}^{(\alpha, \beta)}(x),
\end{align} with the initial conditions $	P_0^{(\alpha, \beta)}(x) = 1, \quad P_{-1}^{(\alpha, \beta)}(x) = 0.$ The recurrence coefficients \( c_{n+1} \) and \( \lambda_{n+1} \) are given by
\begin{align}\label{Recurrence-Coeffcients-Monic-Jacobi}
	\nonumber	\lambda^{(\alpha, \beta)}_{n+1} &= \frac{4n(n + \alpha)(n + \beta)(n + \alpha + \beta)}{(2n + \alpha + \beta)^2 (2n + \alpha + \beta + 1)(2n + \alpha + \beta - 1)},\\
	c^{(\alpha, \beta)}_{n+1} &= \frac{\beta^2 - \alpha^2}{(2n + \alpha + \beta)(2n + \alpha + \beta + 2)}. 
\end{align}
\begin{corollary}\cite[Theorem 1]{Arvesu-Driver-Littlejohn-RamanujanJ-2023}.\label{interlaceJacobi-n_ab-n+1_ab+1}
	Let $P^{(\alpha, \beta)}_n(x)$ be a monic Jacobi  polynomial of degree $n$ and denote the zeros of the polynomial $P^{(\alpha, \beta+1)}_{n+1}(x)$, in ascending order, by $\{x_{k,n+1}\}_{k=1}^{n+1}$. Then, $x_{n+1,n+1}>E_{n,\alpha,\beta}:= -1 + \frac{2(n+1)(n+\alpha+1)}{(2n+\alpha+\beta+2)(2n+\alpha+\beta+3)}$ and $(x-E_{n,\alpha,\beta})P^{(\alpha, \beta)}_n(x)\prec P^{(\alpha, \beta+1)}_{n+1}(x)$ for any $\alpha>-1, \beta >0$ and $n\in \mathbb{N}$. Moreover,   $P^{(\alpha,\beta+1)}_{n+1}(x)\prec P^{(\alpha, \beta)}_n(x)$ for suitable choice of $n\in \mathbb{N}, \alpha>-1, \beta>0$ if and only if the smallest zero of $P^{(\alpha,\beta+1)}_{n+1}(x)$ is greater than the point $E_{n,\alpha,\beta}$ (i.e. $x_{1,n+1}>E_{n,\alpha,\beta}$).
\end{corollary} 
\begin{proof} This result was proved in \cite[Theorem 1]{Arvesu-Driver-Littlejohn-RamanujanJ-2023} and can be shown to be a straightforward consequence of Theorem \ref{MainTheorem1}:
	Let  $\{y_{k,n+1}\}_{k=1}^{n+1}$ denote the zeros of the polynomial $P^{(\alpha, \beta-1)}_{n+1}(x)$. Using (cf. \cite[eq.~(4)]{Arvesu-Driver-Littlejohn-RamanujanJ-2023}), the monic Jacobi polynomial $P^{(\alpha,\beta)}_n(x)$ satisfies the following mixed three-term recurrence relation
	\begin{align}
		\nonumber	&M_{n,\alpha,\beta}\left(x+1+\frac{2 \beta}{2n+\alpha+\beta+3}\right)P^{(\alpha, \beta)}_{n}(x)=-(n+\alpha+\beta+1)_2(x+1)P^{(\alpha, \beta+1)}_{n+1}(x)\\&+\left(x+1 - \frac{2(n+1)(n+\alpha+1)}{(2n+\alpha+\beta+2)(2n+\alpha+\beta+3)}\right)P^{(\alpha, \beta-1)}_{n+1}(x),
	\end{align}
	where $M_{n,\alpha,\beta}=\frac{2(n+1)(n+\alpha+\beta)_2}{(2n+\alpha+\beta+1)_2}$ for $n\in \mathbb{N}, \alpha>-1, \beta>0$. Denoting $\mathcal{G}_{n+1}(x):=P^{(\alpha, \beta+1)}_{n+1}(x)$, $\mathcal{Q}_{n+1}(x):=P^{(\alpha, \beta-1)}_{n+1}(x)$, and $\mathcal{P}_n(x):=P^{(\alpha, \beta)}_{n}(x)$, the zeros of $\mathcal{G}_{n+1}(x)$ and $\mathcal{Q}_{n+1}(x)$ satisfy the interlacing property \eqref{Interlace1withsamedegree} using (cf. \cite[Theorem 2.4]{Driver-Jordaan-Mbuyi-NumericalAlg-2008}) for $\alpha>-1,\beta>0$. Hence, the result directly follows by applying Theorem \ref{MainTheorem1}(1).\end{proof}
	

\begin{corollary}\label{interlaceJacobi-n_ab-n+1_a+1b+1*}
	Let $P^{(\alpha, \beta)}_n(x)$ be a monic Jacobi  polynomial of degree $n$. Let $\{x_{k,n}\}_{k=1}^{n}$ denote the zeros of the polynomial $P^{(\alpha+1, \beta+1)}_{n}(x)$
    and define  \begin{align}
		E_{n,\alpha,\beta}:=  \frac{\alpha-\beta}{2n+\alpha+\beta+2},
	\end{align} for any $\alpha>-1, \beta >-1$ and $n\in \mathbb{N}$. Then $(x-E_{n,\alpha,\beta})P^{(\alpha, \beta)}_n(x)\prec P^{(\alpha+1, \beta+1)}_{n}(x)$. In particular, $P^{(\alpha+1,\beta+1)}_{n}(x)\prec P^{(\alpha, \beta)}_n(x)$, whenever $E_{n,\alpha,\beta}<x_{1,n}$ while  $P^{(\alpha, \beta)}_n(x) \prec P^{(\alpha+1,\beta+1)}_{n}(x)$, whenever $E_{n,\alpha,\beta}>x_{n,n}$ for a suitable choice of $n\in \mathbb{N}, \alpha>-1, \beta>-1$.
\end{corollary} 

\begin{proof}
	It is known \cite[eq.~(8)]{Driver-Jordaan-NumerAlg-2018} that the monic Jacobi polynomial satisfies
	\begin{align}\label{Mixed-TTRR-Monic-Jacobi-DJ-2018-eq8*}
		\frac{n+\alpha+\beta+1}{n}\,P^{(\alpha,\beta)}_n(x)
		&= \frac{2n+\alpha+\beta+1}{n}\,P^{(\alpha+1,\beta+1)}_n(x)
		 - \left(x - \frac{\alpha-\beta}{2n+\alpha+\beta+2}\right)P^{(\alpha+1,\beta+1)}_{n-1}(x).
	\end{align}
	Setting
	\begin{align*}
		\mathcal{G}_n(x) &:= P^{(\alpha+1,\beta+1)}_n(x), &
		\mathcal{Q}_{n-1}(x) &:= P^{(\alpha+1,\beta+1)}_{n-1}(x), &
		\mathcal{P}_n(x) &:= P^{(\alpha,\beta)}_n(x)
	\end{align*}
	and
	\begin{align*}
		E := \frac{\alpha-\beta}{2n+\alpha+\beta+2},
	\end{align*}
	the relation rewrites as
	\begin{align*}
		A(x)\,\mathcal{P}_n(x) = B(x)\,\mathcal{G}_n(x) - (x-E)\,\mathcal{Q}_{n-1}(x)
	\end{align*}
	with positive constants
	\begin{align*}
		A(x) &\equiv \frac{n+\alpha+\beta+1}{n}>0, &
		B(x) &\equiv \frac{2n+\alpha+\beta+1}{n}>0.
	\end{align*}
	Since $P^{(\alpha+1,\beta+1)}_n(x)\prec P^{(\alpha+1,\beta+1)}_{n-1}(x)$, $B(E)>0\neq0$, and $P^{(\alpha, \beta)}_n(x)$   and $P^{(\alpha+1, \beta+1)}_n(x)$  share no common zeros, all assumptions of Theorem~\ref{MainTheorem2*} are satisfied. The asserted interlacing properties follow immediately.
\end{proof}

\begin{remark}\label{rmk1:interlaceJacobi-n_ab-n_a+1b+1}  
	\begin{itemize}
		\item[]
		\item[(i)] A weaker version of Corollary~\ref{interlaceJacobi-n_ab-n+1_a+1b+1*} was previously obtained in \cite[Theorem~3]{Arvesu-Driver-Littlejohn-RamanujanJ-2023}.
		\item[(ii)] Corollary \ref{interlaceJacobi-n_ab-n+1_a+1b+1*} improves \cite[Theorem 3]{Arvesu-Driver-Littlejohn-RamanujanJ-2023} by identifying $E_{n,\alpha,\beta}:=\frac{\alpha-\beta}{2n+\alpha+\beta+2}$ as an additional interlacing point that completes the interlacing of the zeros of $P^{(\alpha, \beta)}_n(x)$ and $P^{(\alpha+1, \beta+1)}_{n}(x)$. Furthermore, we show in Corollary \ref{interlaceJacobi-n_ab-n+1_a+1b+1*} that there is full interlacing between the zeros of $\mathcal{P}_n^{(\alpha,\beta)}$ and $\mathcal{P}_n^{(\alpha+1,\beta+1)}$ if and only if $E_{n,\alpha,\beta}<x_{1,n}$ or $E_{n,\alpha,\beta}>x_{n,n}$.
		\item[(iii)]Theorem~\ref{MainTheorem2*} establishes a more general interlacing criterion for polynomials satisfying a mixed three-term recurrence of the form \eqref{GeneralMixedTTRR2*} and Corollary~\ref{interlaceJacobi-n_ab-n+1_a+1b+1*} is recovered as an immediate special case of Theorem~\ref{MainTheorem2*}.
		
	\end{itemize}
	
\end{remark}

\begin{remark}\label{rmk2:interlaceJacobi-n_ab-n+1_a+1b+1}
	Part of the conclusion of \cite[Theorem~3]{Arvesu-Driver-Littlejohn-RamanujanJ-2023} reads as follows 
	``The remaining two (simple) zeros of $P_{n}^{(\alpha+1,\beta+1)}$ either both lie in one of the intervals with endpoints at a pair of consecutive zeros of $P_{n}^{(\alpha,\beta)}$ or one zero of $P_{n}^{(\alpha+1,\beta+1)}$ lies in the interval $(-1,x_{1,n})$ and one zero of $P_{n}^{(\alpha+1,\beta+1)}$ lies in the interval $(x_{n,n}, 1)$.''  
	It follows from Theorem~\ref{MainTheorem2}  as well as from numerical experiments in Table \ref{tab:zeros_comparison}, that either one zero of $P_{n}^{(\alpha+1,\beta+1)}$ lies in the interval $(-1,x_{1,n})$ and no zero lies in $(x_{n,n}, 1)$, or vice versa, depending on the location of $E_{n,\alpha,\beta}$.  
	This observation indicates that the part of the conclusion of \cite[Theorem~3]{Arvesu-Driver-Littlejohn-RamanujanJ-2023}, which states that one zero of $P_{n}^{(\alpha+1,\beta+1)}$ lies in $(-1,x_{1,n})$ and one zero lies in the interval $(x_{n,n}, 1)$, does not hold. However, we remind the reader that the interlacing follows from Theorem \ref{MainTheorem2*} and Corollary \ref{interlaceJacobi-n_ab-n+1_a+1b+1*}.
\end{remark}

\begin{table}[ht]
	\centering
	\caption{Comparison of the zeros of $P_{n}^{(\alpha,\beta)}$ (denoted by $x_{k,n}$) and $P_{n}^{(\alpha+1,\beta+1)}$ (denoted by $z_{k,n}$). }
	\label{tab:zeros_comparison}
	\footnotesize 
	\setlength{\tabcolsep}{12pt} 
	\begin{tabular}{@{}c cc | cc@{}}
		\toprule
		
		\addlinespace[2pt]
		& \multicolumn{2}{c|}{$n=6, \alpha=2, \beta=14$, $E_{6,\alpha,\beta} = -0.4$} & \multicolumn{2}{c}{$n=7, \alpha=14, \beta=2$, $E_{7,\alpha,\beta} = 0.375$} \\
		\midrule
		$k$ & $x_{k,6}$ & $z_{k,6}$ & $x_{k,7}$ & $z_{k,7}$ \\
		\midrule
		1 & $-0.203565$ & $-0.212298$ & $-0.931498$ & $-0.906419$ \\
		2 & $0.101387$ & $0.0784816$ & $-0.818611$ & $-0.784335$ \\
		3 & $0.369625$ & $0.335892$ & $-0.661375$ & $-0.624494$ \\
		4 & $0.59992$ & $0.560588$ & $-0.465388$ & $-0.431566$ \\
		5 & $0.785274$  & $0.747193$  & $-0.237196$  & $-0.210968$ \\
		6 & $0.918787$  & $ 0.890144$  & $0.017114$  & $0.032615$ \\
		7 & $-$  & $-$  & $0.296953$  & $0.300166$ \\
		\bottomrule
	\end{tabular}
	
	\vspace{0.2cm}
	\begin{minipage}{0.9\textwidth}
		\footnotesize
		\textbf{Note:} Observe that $z_{1,6} < x_{1,6}$ (a zero exists in $(-1, x_{1,6})$), but $z_{6,6} < x_{6,6}$ (no zero in $(x_{6,6}, 1)$). 
		However,  $z_{1,7} > x_{1,7}$ (no zero in $(-1, x_{1,7})$), but $z_{7,7} > x_{7,7}$ (a zero exists in $(x_{7,7}, 1)$).
	\end{minipage}
\end{table}	

In this subsection, numerical computations of the zeros were done using $\text{Mathematica}^{\text{\textregistered}}$ software. 
\subsection{Laguerre polynomials}\label{Examples Laguerre}
Monic Laguerre polynomials $\{L_n^{(\alpha)}(x)\}_{n=0}^{\infty}$ form an orthogonal sequence on the interval $(0, \infty)$ with respect to the measure $d\mu(x) = x^\alpha e^{-x} \, dx,$ where $\alpha > -1$ \cite{Chiharabook}. This sequence satisfies a three-term recurrence relation
\begin{equation} \label{TTRR-Monic-Laguerre}
	L_{n+1}^{(\alpha)}(x) = (x - c_{n+1}^{(\alpha)}) L_n^{(\alpha)}(x) - \lambda_{n+1}^{(\alpha)} L_{n-1}^{(\alpha)}(x),
\end{equation}
with initial conditions $L_0^{(\alpha)}(x) = 1$ and $L_{-1}^{(\alpha)}(x) = 0$. The recurrence coefficients $c_{n+1}^{(\alpha)}$ and $\lambda_{n+1}^{(\alpha)}$ are explicitly given by $	c_{n+1}^{(\alpha)} = 2n + \alpha + 1$  and $ \lambda_{n+1}^{(\alpha)} = n(n + \alpha)$.

\begin{corollary}\cite[Theorem 2.1]{Arvesu-Driver-Littlejohn-ITSF-2021}\label{interlaceLaguerre-n_a-n+1_a+1}
	Let $L^{(\alpha)}_n(x)$ be a monic Laguerre polynomial of degree $n$ with $\a>-1$. Let $\{z_{k,n}\}_{k=1}^{n}$, and  $\{x_{k,n+1}\}_{k=1}^{n+1}$ denote the zeros of the polynomials $L^{(\alpha)}_n(x)$ and $L^{(\alpha+1)}_{n+1}(x)$, respectively, in ascending order.  Then, $x_{n+1,n+1}>n+1$ and $(x-n-1)L^{(\alpha)}_n(x)\prec L^{(\alpha+1)}_{n+1}(x)$ for any $\alpha>-1$ and $n\in \mathbb{N}$. Moreover, $L^{(\alpha+1)}_{n+1}(x)\prec L^{(\alpha)}_n(x)$  for suitable choice of $n\in \mathbb{N}, \alpha>-1$ if and only if the smallest zero of $L^{(\alpha+1)}_{n+1}(x)$ is greater than the point $n+1$ (i.e. $x_{1,n+1}>n+1$).
\end{corollary}
\begin{proof} 
	This result was proved in \cite[Theorem 2.1]{Arvesu-Driver-Littlejohn-ITSF-2021} and is a direct consequence of Theorem \ref{MainTheorem1}: Let  $\{y_{k,n+1}\}_{k=1}^{n+1}$ denote the zeros of the polynomial $L^{(\alpha)}_{n+1}(x)$. Using (cf. \cite[eq.~(1)]{Driver-Jordaan-NumerMath-2007}), the monic Laguerre polynomial $L^{(\alpha)}_n(x)$ satisfies the following mixed three-term recurrence relation
	\begin{align}
		(n+1)(n+\alpha+1)L^{(\alpha)}_n(x)=-xL^{(\alpha+1)}_{n+1}(x)+(x-n-1)L^{(\alpha)}_{n+1}(x),
	\end{align}
	for $n\in \mathbb{N}, \alpha>-1$. Denoting $\mathcal{G}_{n+1}(x):=L^{(\alpha+1)}_{n+1}(x)$, $\mathcal{Q}_{n+1}(x):=L^{(\alpha)}_{n+1}(x)$, and $\mathcal{P}_n(x):=L^{(\alpha)}_n(x)$, the zeros of $\mathcal{G}_{n+1}(x)$ and $\mathcal{Q}_{n+1}(x)$ satisfy \eqref{Interlace1withsamedegree} by \cite[Theorem 2.3]{Driver-Jordaan-NumerMath-2007}. Hence, the result follows immediately from Theorem \ref{MainTheorem1} (1).
\end{proof}

Note that the lower bound $n+1$ of the largest zero $x_{n+1,n+1}$ of $L^{(\alpha+1)}_{n+1}(x)$ is explicitly proved in Corollary \ref{interlaceLaguerre-n_a-n+1_a+1} whereas the lower bound is stated in \cite[Theorem 2.1]{Arvesu-Driver-Littlejohn-ITSF-2021}  without a clear proof.

\section{Proofs of the main results}
\label{Proofs}

{\bf{Proof of Theorem \ref{MainTheorem1}.}} Suppose that the zeros of $\mathcal{G}_{n+1}$ and $\mathcal{Q}_{n+1}$, denoted by $\{x_{k,n+1}\}_{k=1}^{n+1}$ and $\{y_{k,n+1}\}_{k=1}^{n+1}$, respectively, satisfy the alternating property
\begin{align}\label{Interlace1withsamedegree}
	y_{1,n+1}<x_{1,n+1}<y_{2,n+1}<x_{2,n+1}<\cdots<x_{n,n+1}<y_{n+1,n+1}<x_{n+1,n+1}
\end{align}
or 
\begin{align}\label{Interlace2withsamedegree}
	x_{1,n+1}<y_{1,n+1}<x_{2,n+1}<y_{2,n+1}<\cdots<y_{n,n+1}<x_{n+1,n+1}<y_{n+1,n+1},
\end{align}
on the (finite or infinite) interval $(a,b)$. 
Note that $E\neq x_{k,n+1}$ for any $k=1,2,\dots,n+1$, since if $E=x_{k,n+1}$ for some $k$, which is zero of $\mathcal{G}_{n+1}$ then by \eqref{GeneralMixedTTRR}, $x_{k,n+1}$ is the common zero of $\mathcal{G}_{n+1}$ and $\mathcal{P}_n(x)$ and this contradicts the assumption that $\mathcal{G}_{n+1}$ and $\mathcal{P}_n(x)$ have no common zeros.
For each $k=1,2,3,\cdots,n$, evaluating \eqref{GeneralMixedTTRR} at $x_{k,n+1}$ and $x_{k+1,n+1}$, we obtain
\begin{align}\label{GeneralMixedTTRR-Evaluateatx_kandx_k+1}
	\mathcal{P}_n(x_{k,n+1})	\mathcal{P}_n(x_{k+1,n+1})=\frac{(x_{k,n+1}-E)(x_{k+1,n+1}-E)}{A(x_{k,n+1})A(x_{k+1,n+1})} \mathcal{Q}_{n+1}(x_{k,n+1})\mathcal{Q}_{n+1}(x_{k+1,n+1}).
\end{align}
Since the zeros of $\mathcal{G}_{n+1}$ and $\mathcal{Q}_{n+1}$ interlace, we have $\mathcal{Q}_{n+1}(x_{k,n+1})\mathcal{Q}_{n+1}(x_{k+1,n+1})<0$ for each $k=1,2,3,\cdots,n$. Also, by hypothesis, $A(x_{k,n+1})A(x_{k+1,n+1})>0$. Thus, the sign of the left-hand side of \eqref{GeneralMixedTTRR-Evaluateatx_kandx_k+1} depends on the sign of the product $(x_{k,n+1}-E)(x_{k+1,n+1}-E)$, in other words, the location of the point $E$. If $E\not \in (x_{k,n+1}, x_{k+1,n+1})$, then $(x_{k,n+1}-E)(x_{k+1,n+1}-E)>0$ for each $k=1,2,3,\cdots,n$ while $(x_{k,n+1}-E)(x_{k+1,n+1}-E)<0$ if $E\in (x_{k,n+1}, x_{k+1,n+1})$. Therefore, for each $k=1,2,3,\cdots,n$, we have 
\begin{align}\label{product-Pn-negative-in-interval}
	\mathcal{P}_n(x_{k,n+1})	\mathcal{P}_n(x_{k+1,n+1})<0~~  \text{if and only if} ~~E\not \in (x_{k,n+1}, x_{k+1,n+1}).
\end{align}
This shows that the polynomial $\mathcal{P}_n$ has an odd number of zeros in the interval $(x_{k,n+1},x_{k+1,n+1})$ for each $k=1,2,3,\cdots,n$ if it does not contain the point $E$. 
If $E<x_{n+1,n+1}$ then we have two possible sub cases
\begin{itemize}
	\item either $E<x_{1,n+1}$,
	\item or $x_{j,n+1}<E< x_{j+1,n+1}$ for some $j=1,2,\cdots n$.
\end{itemize}
By evaluating \eqref{GeneralMixedTTRR} at $x_{n+1,n+1}$, we get
\begin{align}\label{GeneralMixedTTRR-at-x_n+1}
	\mathcal{P}_n(x_{n+1,n+1})=\frac{(x_{n+1,n+1}-E)}{A(x_{n+1,n+1})}\mathcal{Q}_{n+1}(x_{n+1,n+1})
\end{align}
\begin{enumerate}
	\item  If $E < x_{n+1,n+1}$ and \eqref{Interlace1withsamedegree} holds, then, using the fact that $\lim\limits_{x\rightarrow +\infty}\mathcal{Q}_{n+1}(x)=+\infty$, we have $\mathcal{Q}_{n+1}(x_{n+1,n+1})>0$. Thus, $\mathcal{P}_n(x_{n+1,n+1})>0$ whenever $E < x_{n+1,n+1}$  and \eqref{Interlace1withsamedegree} holds. Hence, the interval $(x_{n+1,n+1}, \infty)$ contains an even number of zeros of $\mathcal{P}_n$.

	\begin{itemize}
		\item When $E < x_{1,n+1}$, by \eqref{product-Pn-negative-in-interval}, each of the $n$ intervals $(x_{k,n+1}, x_{k+1,n+1})$, $k\in\{1,\dots,n\}$ contains exactly one zero of $\mathcal{P}_n$. Hence, $\mathcal{P}_n$ has no zeros in the interval $(x_{n+1,n+1}, \infty)$. Suppose that the zeros of $\mathcal{P}_n$ denoted by $\{z_{i,n}\}_{i=1}^n$. Therefore, the zeros of $\mathcal{P}_n$ and $\mathcal{G}_{n+1}$ must be arranged as  
		\begin{align}
			E < x_{1,n+1} < z_{1,n} < x_{2,n+1} < z_{2,n} < \cdots < x_{n,n+1} < z_{n,n} < x_{n+1,n+1}.
		\end{align}
		\item When $x_{j,n+1}<E< x_{j+1,n+1}$ for some $j=1,2,\cdots n$,  \eqref{GeneralMixedTTRR-Evaluateatx_kandx_k+1} ensures that $\mathcal{P}_n(x_{j,n+1})	\mathcal{P}_n(x_{j+1,n+1})>0$.  This means that  $(x_{j,n+1}, x_{j+1,n+1})$ contains even number of zeros of  $\mathcal{P}_n$. Since $n-1$ zeros of $\mathcal{P}_n$ are already captured in $(x_{k,n+1}, x_{k+1,n+1})$ for each $k=1,2,3, \cdots n$ and $k\neq j$. Thus, the interval $(x_{k,n+1}, x_{k+1,n+1})$ and $(x_{n+1,n+1}, \infty)$ does not contains the zero of $\mathcal{P}_n$. Therefore the only possibilities of the remaining one zero of $\mathcal{P}_n$ lies in the interval $(-\infty, x_{1,n+1})$. The possible arrangement of the zeros is
		\begin{align}
			\nonumber	z_{1,n}<x_{1,n+1}<z_{2,n}<x_{2,n+1}<z_{3,n}<\cdots< x_{j,n+1}<E<x_{j+1,n+1}
			\\<z_{j+1,n}<\cdots<x_{n,n+1}<z_{n,n}<x_{n+1,n+1}.
		\end{align}
		Hence, $(x-E)\mathcal{P}_n(x)\prec\mathcal{G}_{n+1}(x)$.
	\end{itemize}
	If  $E>x_{n+1,n+1}$ and  \eqref{Interlace1withsamedegree} holds, then by using \eqref{GeneralMixedTTRR-at-x_n+1}, we have  $\mathcal{P}_n(x_{n+1,n+1})<0$. This suggest that $\mathcal{P}_n$ has odd number of zeros inside the interval $(x_{n+1,n+1}, \infty)$. However, \eqref{product-Pn-negative-in-interval} ensures that exactly one zero of the $n$ degree polynomial $\mathcal{P}_n$ lie in each interval $(x_{k,n+1}, x_{k+1,n+1})$ for $k=1,2,3, \cdots n$. This contradicts the fact that $\mathcal{P}_n$ must have an odd number of zeros inside the interval $(x_{n+1,n+1}, \infty)$. Therefore, the point $E$ cannot be greater than the largest zero of the polynomial $\mathcal{G}_{n+1}$, whenever \eqref{Interlace1withsamedegree} holds. 	
	\item If $E < x_{n+1,n+1}$ and \eqref{Interlace2withsamedegree} holds, then using the fact that $\lim\limits_{x\rightarrow +\infty}\mathcal{Q}_{n+1}(x)=+\infty$, we have $\mathcal{Q}_{n+1}(x_{n+1,n+1})<0$. Thus, $\mathcal{P}_n(x_{n+1,n+1})<0$, whenever $E < x_{n+1,n+1}$. Hence, at least one zero or an odd number of zeros of $\mathcal{P}_n$ lie in the interval $(x_{n+1,n+1}, \infty)$.
	\begin{itemize}
		\item When $E<x_{1,n+1}$, by \eqref{GeneralMixedTTRR-Evaluateatx_kandx_k+1}, each interval $(x_{k,n+1},x_{k+1,n+1})$ contains exactly one zero of $\mathcal{P}_n$. Thus, no zero of $\mathcal{P}_n$ lies in the interval $(x_{n+1,n+1}, \infty)$, which is a contradiction to the argument that there is an odd number of zeros of $\mathcal{P}_n$ contained in the interval $(x_{n+1,n+1}, \infty)$. Therefore, the point $E$ cannot be less than the smallest zero of the polynomial $\mathcal{G}_{n+1}$ whenever \eqref{Interlace2withsamedegree} holds. 
		
		\item When $x_{j,n+1}<E< x_{j+1,n+1}$ for some $j=1,2,\cdots n$,  \eqref{GeneralMixedTTRR-Evaluateatx_kandx_k+1} ensures that \newline $\mathcal{P}_n(x_{j,n+1})	\mathcal{P}_n(x_{j+1,n+1})>0$.  This means that  $(x_{j,n+1}, x_{j+1,n+1})$ contains even number of zeros of  $\mathcal{P}_n$. Since $n-1$ zeros of $\mathcal{P}_n$ are already captured in $(x_{k,n+1}, x_{k+1,n+1})$ for each $k=1,2,3, \cdots n$ and $k\neq j$. Thus, the interval $(x_{k,n+1}, x_{k+1,n+1})$  does not contain the zero of $\mathcal{P}_n$ and $(x_{n+1,n+1}, \infty)$ contain exactly one zero of the $\mathcal{P}_n$ . Therefore, the only possibility for the arrangement of the zero is  
		\begin{align}
			\nonumber	x_{1,n+1}<z_{1,n}<x_{2,n+1}<z_{2,n}<\cdots< x_{j,n+1}<E<x_{j+1,n+1}<z_{j,n}
			\\<\cdots<x_{n,n+1}<z_{n-1,n}<x_{n+1,n+1}<z_{n,n}.
		\end{align}
		Hence, $\mathcal{G}_{n+1}(x)\prec(x-E)\mathcal{P}_n(x)$.
	\end{itemize}
	If  $E>x_{n+1,n+1}$ and  \eqref{Interlace2withsamedegree} holds, then by using \eqref{GeneralMixedTTRR-at-x_n+1}, we have  $\mathcal{Q}_{n+1}(x_{n+1,n+1})<0$. Thus $\mathcal{P}_{n}(x_{n+1,n+1})>0$. This suggests that $\mathcal{P}_n$ has either no zero or an even number of zeros inside the interval $(x_{n+1,n+1}, \infty)$. However, \eqref{product-Pn-negative-in-interval} ensures that exactly one zero of the $n$ degree polynomial $\mathcal{P}_n$ lie in each interval $(x_{k,n+1}, x_{k+1,n+1})$ for $k=1,2,3, \cdots n$. Therefore, $\mathcal{P}_n$ has no zero inside the interval $(x_{n+1,n+1}, \infty)$. The only possibility of the location of zeros of $\mathcal{P}_n$ and $\mathcal{Q}_{n+1}$ is 
	\begin{align}
		x_{1,n+1}<z_{1,n}<x_{2,n+1}<z_{2,n}<\cdots<x_{n,n+1}<z_{n,n}<x_{n+1,n+1}<E.
	\end{align}
	
	Hence we proved that $\mathcal{G}_{n+1}\prec \mathcal{P}_n$.
\end{enumerate}This completes the proof.

\vspace{1cm}

{\bf{Proof of Theorem \ref{MainTheorem2*}.}} Let $\{x_{k,n}\}_{k=1,n}$
$\{z_{k,n}\}_{k=1}^{n}$, and $\{y_{k,n-1}\}_{k=1}^{n-1}$ denote the zeros of $\mathcal{G}_n$, $\mathcal{P}_n$ and $\mathcal{Q}_{n-1}$, respectively.
Assume that the zeros of $\mathcal{G}_{n}$ and $\mathcal{Q}_{n-1}$ strictly interlace in a (finite or infinite) interval $(a,b)$, i.e.,
\begin{align}\label{interlace-assumption-MainTheorem2*}
	x_{1,n} < y_{1,n-1} < x_{2,n} < y_{2,n-1}< \cdots < y_{n-1,n-1} < x_{n,n}.
\end{align} Note that $E\neq x_{k,n}$ for any $k=1,2,\dots,n$, since if $E=x_{k,n}$ for some $k$, which is zero of $\mathcal{G}_{n}$ then by \eqref{GeneralMixedTTRR2*}, $x_{k,n}$ is the common zero of $\mathcal{G}_{n}$ and $\mathcal{P}_n$ and this contradicts the assumption that $\mathcal{G}_{n}$ and $\mathcal{P}_n$ have no common zeros.
For each $k=1,2,\dots,n-1$, evaluate \eqref{GeneralMixedTTRR2*} at the points $x_{k,n}$ and $x_{k+1,n}$. This yields
\begin{align}\label{GeneralMixedTTRR2-Evaluateatx_kandx_k+1*}
	\mathcal{P}_n(x_{k,n})\mathcal{P}_n(x_{k+1,n})
	=\frac{(x_{k,n}-E)(x_{k+1,n}-E)}{A(x_{k,n})A(x_{k+1,n})}\,\mathcal{Q}_{n-1}(x_{k,n})\mathcal{Q}_{n-1}(x_{k+1,n}).
\end{align}
Since the zeros of $\mathcal{G}_n$ and $\mathcal{Q}_{n-1}$ interlace, 
$\mathcal{Q}_{n-1}(x_{k,n})\mathcal{Q}_{n-1}(x_{k+1,n}) < 0$ 
for each $k=1,2,\dots,n-1$. 
Moreover, $A(x_{k,n})A(x_{k+1,n}) > 0$ because $A(x)>0$ on $(a,b)$. 
Therefore,
\begin{align}\label{Thm2-NS condition-sign-Pn-interval*}
	\mathcal{P}_n(x_{k,n})\mathcal{P}_n(x_{k+1,n}) > 0
	\quad \text{if and only if} \quad 
	x_{k,n} < E < x_{k+1,n}.
\end{align}
Consequently, $\mathcal{P}_n$ changes sign across the interval $(x_{k,n},x_{k+1,n})$ if and only if this interval does not contain $E$.

Evaluating \eqref{GeneralMixedTTRR2*} at $x = x_{n,n}$ yields
\begin{align}\label{GeneralMixedTTRR2-at-x_nn*}
	\mathcal{P}_n(x_{n,n}) = -\frac{(x_{n,n}-E)}{A(x_{n,n})}\,\mathcal{Q}_{n-1}(x_{n,n}).
\end{align}
The interlacing condition \eqref{interlace-assumption-MainTheorem2*} implies that $\mathcal{Q}_{n-1}$ has no zero in $(x_{n,n},b)$ and $\mathcal{Q}_{n-1}(x_{n,n}) > 0$. Since $A(x_{n,n}) > 0$, 
\begin{align*}
	\mathcal{P}_n(x_{n,n}) < 0 \quad \text{if} \quad E < x_{n,n}, \qquad 
	\mathcal{P}_n(x_{n,n}) > 0 \quad \text{if} \quad E > x_{n,n}.
\end{align*}
Thus the interval $(x_{n,n},b)$ contains an odd number of zeros of $\mathcal{P}_n$ when $E < x_{n,n}$ and an even number (possibly zero) when $E > x_{n,n}$.

If $E < x_{1,n}$ or $E > x_{n,n}$, then $E$ lies outside all intervals $(x_{k,n},x_{k+1,n})$, $k=1,\dots,n-1$. Hence \eqref{Thm2-NS condition-sign-Pn-interval*} forces a sign change of $\mathcal{P}_n$ in each of these $n-1$ intervals, so $\mathcal{P}_n$ has exactly one zero in each $(x_{k,n},x_{k+1,n})$.  
When $E < x_{1,n}$ the interval $(x_{n,n},b)$ contains an odd number of zeros, while $(a,x_{1,n})$ contains an even number; the remaining zero therefore lies in $(x_{n,n},b)$.  
When $E > x_{n,n}$ the situation is reversed: $(a,x_{1,n})$ contains an odd number and $(x_{n,n},b)$ an even number, so the remaining zero lies in $(a,x_{1,n})$.  
In both situations,  $(x-E)\mathcal{P}_n\prec \mathcal{G}_n$ since
\begin{align*}
	z_{1,n} < x_{1,n} < z_{2,n} < \cdots < z_{n,n} < x_{n,n}	~\text{or}~	x_{1,n} < z_{1,n} < x_{2,n}  < \cdots < z_{n-1,n} < x_{n,n} < z_{n,n},
\end{align*}
according to whether $E > x_{n,n}$ or $E < x_{1,n}$ respectively.

If $x_{j,n} < E < x_{j+1,n}$ for some $j \in \{1, 2, \ldots, n-1\}$, then, by \eqref{Thm2-NS condition-sign-Pn-interval*}, at least one zero or odd number of zero of $\mathcal{P}_n$ lie in each interval $(x_{k,n}, x_{k+1,n})$ with $k \neq j$. Also, by \eqref{GeneralMixedTTRR2-at-x_nn*}, at least one zero or odd number of zero of  $\mathcal{P}_n$ lie in the interval $(x_{n,n},b)$. Thus, counting the number of zeros of $\mathcal{P}_n$, exactly one zero lie in each interval $(x_{k,n}, x_{k+1,n})$ with $k \neq j$ and $(x_{n,n},b)$. Thus, only possibility of the remaining zero of $\mathcal{P}_n$ lie in $(a,x_{1,n})$. This shows that $(x-E)\mathcal{P}_n(x)\prec \mathcal{G}_n(x)$; that is,
\begin{align*}
	z_{1,n} < x_{1,n} < z_{2,n} < \cdots < x_{j,n} < E  < x_{j+1,n} < \cdots < x_{n,n} < z_{n,n}.
\end{align*}

This completes the proof.
\vspace{1cm} 

{\bf{Proof of Theorem \ref{MainTheorem2}.}}
Suppose that the zeros of the monic polynomials $\mathcal{G}_{n}$, $\mathcal{P}_n$, and $\mathcal{Q}_{n+1}$ are denoted by 
$\{x_{k,n}\}_{k=1}^{n}$, $\{z_{k,n}\}_{k=1}^{n}$, and $\{y_{k,n+1}\}_{k=1}^{n+1}$, respectively. 
Assume that $\mathcal{Q}_{n+1}\prec\mathcal{G}_{n}$, i.e.,
\begin{align}
	y_{1,n+1} < x_{1,n} < y_{2,n+1} < x_{2,n} < \cdots < y_{n,n+1} < x_{n,n} < y_{n+1,n+1},
\end{align}
on a (finite or infinite) interval $(a,b)$.
For each \( k = 1, 2, 3, \ldots, n-1 \), we evaluate \eqref{GeneralMixedTTRR2} at points \( x_{k,n} \) and \( x_{k+1,n} \). This gives us the following equation
\begin{align}\label{GeneralMixedTTRR2-Evaluateatx_kandx_k+1}
	\mathcal{P}_n(x_{k,n})	\mathcal{P}_n(x_{k+1,n})=\frac{(x_{k,n}-E)(x_{k+1,n}-E)}{A(x_{k,n})A(x_{k+1,n})} \mathcal{Q}_{n+1}(x_{k,n})\mathcal{Q}_{n+1}(x_{k+1,n}).
\end{align}
Since the zeros of \( \mathcal{G}_{n} \) and \( \mathcal{Q}_{n+1} \) interlace, we have \( \mathcal{Q}_{n+1}(x_{k,n}) \mathcal{Q}_{n+1}(x_{k+1,n}) < 0 \) for each \newline \( k = 1, 2, 3, \ldots, n-1 \). Moreover, by assumption, \( A(x_{k,n}) A(x_{k+1,n}) > 0 \). Thus, for each \( k = 1, 2, \ldots, n-1 \), we conclude that
\begin{align}\label{Thm2-NS condition-sign-Pn-interval}
	\mathcal{P}_n(x_{k,n})	\mathcal{P}_n(x_{k+1,n})>0~ \text{if and only if}  ~E \in (x_{k,n},x_{k+1,n}).
\end{align}

Further, evaluating \eqref{GeneralMixedTTRR2} at $x_{n,n}$ yields 
\begin{align}\label{GeneralMixedTTRR2-at-x_nn}
	\mathcal{P}_n(x_{n,n}) = -\frac{(x_{n,n}-E)}{A(x_{n,n})}\,\mathcal{Q}_{n+1}(x_{n,n}).
\end{align}
Since $\mathcal{Q}_{n+1}(x_{n,n}) < 0$, the sign of $\mathcal{P}_n(x_{n,n})$ depends on the position of $E$. If $E < x_{n,n}$, then $\mathcal{P}_n(x_{n,n}) > 0$, implying that $(x_{n,n}, b)$ contains an even number of zeros of $\mathcal{P}_n$ . To locate all zeros of $\mathcal{P}_n$, we consider the following subcases.

\smallskip
\noindent
\textbf{Subcase 1.} If $E < x_{1,n}$, then by \eqref{Thm2-NS condition-sign-Pn-interval}, each of the $n-1$ intervals $(x_{k,n}, x_{k+1,n})$, $k=1,2, \dots,n-1$ contains at least one zero of $\mathcal{P}_n$, and counting the number of zeros, exactly one zero of $\mathcal{P}_n$ for $k = 1, 2, \ldots, n-1$. Consequently, the remaining zero of $\mathcal{P}_n$ must lie in $(a, x_{1,n})$. Denoting by $\{z_{k,n}\}_{k=1}^{n}$ the zeros of $\mathcal{P}_n$, it follows that $\mathcal{P}_n\prec\mathcal{G}_n$, i.e.,
\begin{align*}
	z_{1,n} < x_{1,n} < z_{2,n} < x_{2,n} < \cdots < z_{n,n} < x_{n,n}.
\end{align*}

\smallskip
\noindent
\textbf{Subcase 2.} If $x_{k',n} < E < x_{k'+1,n}$ for some $k' \in \{1, 2, \ldots, n-1\}$, then, by \eqref{Thm2-NS condition-sign-Pn-interval}, for each
$k\neq k'$, the interval $(x_{k,n},x_{k+1,n})$ contains an odd number of zeros of $\mathcal{P}_n$, and hence at least one zero of $\mathcal{P}_n$. Since there are $n-2$ such intervals, it follows that at least $n-2$ zeros of $\mathcal{P}_n$ already lie in distinct intervals whose end points are consecutive zeros of $\mathcal{G}_n$.
As $\deg(\mathcal{P}_n)=n$, only two zeros of $\mathcal{P}_n$ remain to be located. When counting the number of zeros of $\mathcal{P}_n$, there are two possible configurations for the remaining two zeros. Either there exists a fixed index $t \neq k'$ such that exactly one  interval $(x_{t,n}, x_{t+1,n})$ contains three zeros of $\mathcal{P}_n$, while each of the remaining $n-3$ intervals $(x_{k,n}, x_{k+1,n})$ (for $k = 1,\dots,n-1$, $k \neq k'$, $k \neq t$) contains exactly one zero; or each of the $n-2$ intervals $(x_{k,n}, x_{k+1,n})$ with $k \neq k'$ contains exactly one zero, and the remaining two zeros lie in one of the following: $(a, x_{1,n})$, $(x_{k',n}, x_{k'+1,n})$, or $(x_{n,n}, b)$. Moreover, since \(E<x_{n,n}\), it follows from \((4.17)\) that \(\mathcal{P}_n(x_{n,n})>0\), and therefore the interval \((x_{n,n},b)\) contains an even number of zeros of \(\mathcal{P}_n\).
Hence the remaining two zeros cannot split with one lying in \((a,x_{1,n})\) and the other in \((x_{n,n},b)\). If the remaining two zeros lie in $(x_{k',n}, x_{k'+1,n})$ such that one lies in $(x_{k',n}, E)$ and the other in $(E, x_{k'+1,n})$, then $(x-E)\mathcal{G}_n(x)\prec \mathcal{P}_n(x)$; that is,
\begin{align*}
	x_{1,n} < z_{1,n} < x_{2,n} < z_{2,n} < \cdots < x_{k',n} < z_{k',n} < E < z_{k'+1,n} < x_{k'+1,n} < \cdots < z_{n,n} < x_{n,n}.
\end{align*}

\smallskip
\noindent
On the other hand, if $E > x_{n,n}$, then from \eqref{GeneralMixedTTRR2-at-x_nn} we have $\mathcal{P}_n(x_{n,n}) < 0$, so $\mathcal{P}_n$ has an odd number of zeros in $(x_{n,n}, b)$. Moreover, by \eqref{Thm2-NS condition-sign-Pn-interval}, each interval $(x_{k,n}, x_{k+1,n})$, $k = 1, 2, \ldots, n-1$, contains exactly one zero of $\mathcal{P}_n$. Hence, the only possible configuration of zeros is
\begin{align}
	x_{1,n} < z_{1,n} < x_{2,n} < z_{2,n} < x_{3,n} < \cdots < x_{n,n} < z_{n,n}.
\end{align}
This completes the proof.

\section{Conclusion}
The results obtained in this study provide a systematic analysis for completing the interlacing of zeros of distinct polynomial sequences by involving an additional interlacing point, $E$, that appears naturally in the general mixed recurrence relations used to prove the results.
 
The general mixed recurrence relations introduced here serve as a versatile tool for obtaining new interlacing results. Because the main findings of this work are independent of orthogonality constraints, the framework is equally effective for both orthogonal and non-orthogonal sequences. This is demonstrated through novel interlacing results derived for Krawtchouk, Meixner, and Narayana polynomials. We anticipate that the generality of our framework will facilitate the discovery of new interlacing properties across a much wider family of polynomials.
 
Beyond new examples, the analysis also reinforces existing literature for Laguerre and Jacobi polynomials. Our general approach provides straightforward proofs for the interlacing properties of classical Jacobi (see Corollaries \ref{interlaceJacobi-n_ab-n+1_ab+1} and \ref{interlaceJacobi-n_ab-n+1_a+1b+1*}) and Laguerre polynomials (see Corollary \ref{interlaceLaguerre-n_a-n+1_a+1}), effectively recovering the results presented in \cite[Theorem 1]{Arvesu-Driver-Littlejohn-RamanujanJ-2023}, \cite[Theorem 3]{Arvesu-Driver-Littlejohn-RamanujanJ-2023}, and \cite[Theorem 2.1]{Arvesu-Driver-Littlejohn-ITSF-2021}. More importantly, by precisely evaluating the conditions imposed on $E$, our framework improves the interlacing results established in \cite[Theorem 3]{Arvesu-Driver-Littlejohn-RamanujanJ-2023} (see Corollary \ref{interlaceJacobi-n_ab-n+1_a+1b+1*} and Remark   \ref{rmk1:interlaceJacobi-n_ab-n_a+1b+1} and \ref{rmk2:interlaceJacobi-n_ab-n+1_a+1b+1}). We also prove full interlacing of zeros of these polynomials under assumption of appropriate position(s) for the extra point, $E$.


\bibliographystyle{cas-model2-names}

\bibliography{cas-refs}



\end{document}